\documentclass[12pt,a4paper]{article}

\usepackage[multiple]{footmisc}
\usepackage{amsmath,amsfonts,url,amssymb,amsthm,tipa,enumerate,verbatim,threeparttable,multirow,hyperref}

\newcommand{\C}{{\mathbb{C}}}
    \renewcommand{\c}{{\mathbf{c}}}
\newcommand{\B}{{\mathcal{B}}}
\newcommand{\Div}{{\text{Div}}}
    \renewcommand{\div}{{\text{div}}}

\newcommand{\F}{{\mathbb{F}}}
\newcommand{\FF}{{\mathcal{F}}}
\newcommand{\e}{{\mathbf{e}}}

\newcommand{\iP}{{\mathfrak{P}}}

\newcommand{\LP}{{\mathcal{P}}}
    \renewcommand{\mod}{\ \mathrm{mod}\ }

\newcommand{\Norm}{{\textnormal{Norm}}}
    \renewcommand{\O}{{\mathcal{O}}}

\newcommand{\Q}{{\mathbb{Q}}}
\newcommand{\R}{{\mathbb{R}}}
    
\newcommand{\rank}{{\text{\em rank\ }}}
\newcommand{\supp}{{\text{supp}}}

\newcommand{\vol}{{\text{vol}}}
    \newcommand{\bu}{{\mathbf{u}}}
    \newcommand{\bv}{{\mathbf{v}}}

\newcommand{\Z}{{\mathbb{Z}}}

\newcommand\norm[1]{\left\lVert#1\right\rVert}
\newcommand{\Magma}{{\sf Magma}}

\newtheorem{theorem}{Theorem}[section]
\newtheorem{lemma}[theorem]{Lemma}
\newtheorem{proposition}[theorem]{Proposition}

\theoremstyle{definition}

\theoremstyle{remark}
\newtheorem{remark}[theorem]{Remark}

\numberwithin{equation}{section}

\urldef{\DOI}\url{http://dx.doi.org/10.1016/j.ffa.2015.07.005}

\begin{document}




\title{Improvement on Asymptotic Density of Packing Families Derived from Multiplicative Lattices\thanks{The final publication is available at Elsevier via \DOI.}}


\author{Shantian Cheng\footnote{The author has been supported by NTU Research Scholarship at Nanyang Technological University.}\\\small Division of Mathematical Sciences,\\ \small
School of Physical and Mathematical Sciences,\\ \small
Nanyang Technological University,\\ \small
SPMS-MAS-04-20, 21 Nanyang Link,
637371 Singapore\\\small scheng002@e.ntu.edu.sg; chengshantian@gmail.com}
\date{}
\maketitle
\begin{abstract}
Let $\omega=(-1+\sqrt{-3})/2$. For any lattice $P\subseteq \Z^n$, $\LP=P+\omega P$ is a subgroup of $\O_K^n$, where $\O_K=\Z[\omega]\subseteq \C$. As $\C$ is naturally isomorphic to $\R^2$, $\LP$ can be regarded as a lattice in $\R^{2n}$. Let $P$ be a multiplicative lattice (principal lattice or congruence lattice) introduced by Rosenbloom and Tsfasman. We concatenate a family of special codes with $t_{\iP}^\ell\cdot(P+\omega P)$, where $t_{\iP}$ is the generator of a prime ideal $\iP$ of $\O_K$. Applying this concatenation to a family of principal lattices, we obtain a new family with asymptotic density exponent $\lambda\geqslant-1.26532182283$, which is better than $-1.87$ given by Rosenbloom and Tsfasman considering only principal lattice families.
For a new family based on congruence lattices, the result is $\lambda\geqslant -1.26532181404$, which is better than $-1.39$ by considering only congruence lattice families.

\noindent
{\bf Keywords:} Concatenation; multiplicative lattices; special number field; Gilbert-Varshamov bound

\noindent
{\bf MSC2010:} 52C17; 11R58; 94B65; 14H05

\end{abstract}



\section{Introduction}

Sphere packing is a classical problem on how to pack non-overlapping equal spheres densely in $\R^N$. Many methods and results from different disciplines, such as discrete geometry, combinatorics, number theory and coding theory, etc. have been involved in this problem. For a detailed survey on the development in this territory, the reader may refer to the book of Conway and Sloane \cite{sloane1999sphere}.

Sphere packing evolves into two concrete problems. One is how to construct packings of larger density than the record (e.g. \cite[Table 1.2-1.3]{sloane1999sphere}) in Euclidean spaces of specific dimension $N$. Another one is how to construct families of packings with dimension $N\rightarrow \infty$ such that the asymptotic density exponent has small absolute value.

Minkowski gave a nonconstructive bound that there exists one packing family $\FF$ such that the asymptotic density exponent $\lambda(\FF)\geqslant -1$ (See \cite[p.184]{cassels1997introduction}). However, it is a challenge to construct families with $\lambda(\FF)<\infty$ explicitly (such families are called {\em asymptotically good}). The known constructive bounds for families with polynomial or exponential construction complexity in terms of $N$ are listed in the book of Litsyn and Tsfasman \cite[p.628]{tsfasman1991algebraic}. To our best knowledge, they still remain the best so far.

One classical packing construction idea is to concatenate proper codes with special packings in $\Z^n$. This method may offer new packings denser than the original ones. There are five well-known constructions based on this idea, which are referred as Construction A,B,C(due to Leech and Sloane); D(due to Bos, Conway and Sloane); E(due to Barnes and Sloane). More details about these constructions can be found in \cite{sloane1999sphere,rush2004spherepacking,zong1999sphere}.

Particularly, in Construction C \cite[Chapter 5]{zong1999sphere}, the binary expansion of the coordinates in $\Z^n$ is considered. A point is a packing center if and only if the first $\ell$ coordinate arrays are codewords in $\ell$ certain binary codes respectively.
Subsequently, instead of packings in $\Z^n$, Xing \cite{xing2008dense} considered the packings in $\O_K^n$, where $\O_K$ denotes the ring of integers in number field $K=\Q(\sqrt{-3})$, and then replaced the binary expansion by $\iP$-adic expansion, where $\iP$ is a nonzero prime ideal of $\O_K$. He offered several packing constructions with the best-known densities in small dimensions and obtained an unconditional bound of asymptotic density exponent $\lambda\geqslant -1.2653$.

For the asymptotic density exponent, Xing \cite{xing2008dense} concatenated $\ell$ codes with a packing $\LP^{(N)}$ in $\O_K^N$ of fixed minimum Euclidean distance. When $N$ tends to $\infty$, the number of codes $\ell\rightarrow \infty$, and the family $\left\{\LP^{(N)}\right\}$ is not asymptotically good. However, the resulting packing family is asymptotically good.

In this paper, we further explore the concatenating method of Xing to obtain another method to construct asymptotically good packing families. Compared with Xing's construction, we employ asymptotically good packing families and concatenate finitely many codes to them. The number of codes remains finite though the dimension $N\rightarrow \infty$.

Explicitly, we apply the generalized concatenating method to Rosenbloom and Tsfasman's multiplicative lattices in function fields (see \cite{Multiplicative1990}), and we get two asymptotically good families with bounds $\lambda\geqslant -1.26532182282$ (principal lattice case) and $\lambda\geqslant -1.26532181404$ (congruence lattice case), while the bounds for multiplicative lattice families provided in \cite{Multiplicative1990} are $-1.87$ and $-1.39$ respectively. Hence our construction improves the asymptotic density of packing families derived from multiplicative lattices.

In Section \ref{preliminaries}, we recall some basic knowledge of sphere packing, coding theory and concatenation based on the number field $K=\Q(\sqrt{-3})$. In Section~\ref{sec:remark on Xing}, we give some remarks on Xing's construction in comparison with the basic concatenation with $\O_K^n$. The general description of our new construction comes in Section \ref{sec:our general method}, and as an application, we apply the new method on the multiplicative lattices in Section \ref{concatenation with multiplicative lattices}. In Section \ref{sec:comparison} and Section \ref{sec:conclusion}, we compare the results and conclude our contribution.

\section{Preliminaries}\label{preliminaries}

\subsection{Sphere Packing}
Let $\LP$ be the set of centers of packed spheres and $\B_N(R)$ be the set
\begin{eqnarray*}
  \left\{\left(a_1,\cdots,a_N\right)\in \R^N: \sqrt{a_1^2+\cdots+a_N^2}\leqslant R\right\}.
\end{eqnarray*}
As a sphere packing construction is uniquely determined by the arrangement of the sphere centers, we also use $\LP$ to denote the corresponding packing.

For a packing $\LP$, the radius of the equal packed spheres is $d_E(\LP)/2$, where $d_E(\LP)$ is the minimum Euclidean distance between two distinct points in $\LP$.
Then the density $\Delta(\LP)$ of packing $\LP$ is defined as
\begin{eqnarray*}
  \Delta(\LP)=\limsup_{R\rightarrow \infty}\dfrac{\left|\LP\cap \B_N(R)\right|\cdot \left(d_E\left(\LP\right)/2\right)^N\cdot V_N}{\vol\left(\B_N\left(R+d_E\left(\LP\right)/2\right)\right)},
\end{eqnarray*}
where $V_N$ is the volume of the unit sphere in $\R^N$, that is
\begin{eqnarray*}
  V_N=\begin{cases}
    \dfrac{\pi^{N/2}}{\left(N/2\right)!},&\text{if $N$ is even;}\\
   \dfrac{2^N\pi^{\left(N-1\right)/2}\left(\left(N-1\right)/2\right)!}{N!},& \text{if $N$ is odd.}
  \end{cases}
\end{eqnarray*}
The sphere packing problem is to construct packings obtaining large density $\Delta(\LP)$. Moreover, the center density $\delta(\LP)$ and density exponent $\lambda(\LP)$ are defined respectively as
\[\delta(\LP)=\dfrac{\Delta(\LP)}{V_N},\quad \lambda(\LP)=\dfrac 1N \log_2 \Delta(\LP).\]

If $\LP=L$ forms a lattice, the density of lattice packing $L$ can be simplified as
\begin{eqnarray*}
  \Delta(L)=\dfrac{\left(d_E(L)/2\right)^NV_N}{\det(L)},
\end{eqnarray*}
where $\det(L)$ is the determinant of $L$.

 When we explore the asymptotic behavior of a packing family $\FF=\left\{\LP^{(N)}\right\}$ as dimension $N$ tends to $\infty$, we consider the asymptotic density exponent of the family
\begin{eqnarray*}
  \lambda(\FF)=\limsup_{N\rightarrow \infty}\dfrac 1N \log_2 \Delta\left(\LP^{(N)}\right).
\end{eqnarray*}
Note that by Stirling formula, as $N\rightarrow \infty$, we have
\begin{eqnarray*}\label{Stirling}
  \log_2 V_N=-\dfrac N2\log_2\dfrac {N}{2\pi e}-\dfrac 12\log_2(N\pi)-\epsilon,
\end{eqnarray*}
where $0<\epsilon<(\log_2 e)/(6N)$.

\subsection{Coding theory}\label{coding theory}

We recall some notations and results in coding theory.

For a $q$-ary code $C$, let $n(C),M(C)$ and $d_H(C)$ denote the length, the size, and the minimum Hamming distance of $C$, respectively. Such code is usually referred to as an $\left(n(C),M(C),d_H(C)\right)$-code.
Moreover, the relative minimum distance $\varrho(C)$ and the rate $R(C)$ are defined respectively as
\begin{eqnarray*}
\varrho(C)=\frac{d_H(C)}{n(C)},\quad  R(C)=\frac{\log_q M(C)}{n(C)}.
\end{eqnarray*}

Let $U_q$ be the set of the ordered pair $(\varrho,R)\in\R^2$ for which there exists a family $\{C_i\}_{i=0}^\infty$ of $q$-ary codes with $n(C_i)$ increasingly goes to $\infty$ as $i$ tends to $\infty$ and
\begin{eqnarray*}
  \varrho=\lim_{i\rightarrow \infty}\varrho(C_i),\quad \text{and}\quad R=\lim_{i\rightarrow \infty}R(C_i).
\end{eqnarray*}
Here is a result on $U_q$:

\begin{proposition}[{\cite[Section 1.3.1]{tsfasman1991algebraic}} or {\cite[Proposition 3.1]{xing2008dense}}]
  There exists a continuous function $R_q(\varrho)$, $\varrho\in[0,1]$, such that
  \[U_q=\left\{(\varrho,R)\in \R^2:0\leqslant R\leqslant R_q(\varrho),\ 0\leqslant \varrho\leqslant 1 \right\}.\]
  Moreover, $R_q(0)=1$, $R_q(\varrho)=0$ for $\varrho\in\left[(q-1)/q,1\right]$, and $R_q(\varrho)$ decreases on the interval $[0,(q-1)/q]$.
\end{proposition}
For $0<\varrho<1$, the $q$-ary entropy function is given as \begin{eqnarray*}
  H_q(\varrho)=\varrho\log_q(q-1)-\varrho\log_q\varrho-(1-\varrho)\log_q(1-\varrho).
\end{eqnarray*}
The asymptotic Gilbert-Varshamov (GV) bound indicates that
\begin{eqnarray}\label{eqn:GV bound}
  R_q(\varrho)\geqslant R_{GV}(q,\varrho):=1-H_q(\varrho),\quad\text{for all }\varrho\in\left(0,\dfrac {q-1}{q}\right).
\end{eqnarray}
Moreover, for any given rate $R$, there exists a family of linear codes which meets the GV bound (see \cite[Section 17.7]{macwilliams1977theory}).


\subsection{Concatenation based on number field $K=\Q\left(\sqrt{-3}\right)$}

 The concatenation based on $\Q(\sqrt{-3})$ has been explained in \cite{xing2008dense}. We recall some key properties first.

Let $\omega=(-1+\sqrt{-3})/2$, $K=\Q(\sqrt{-3})$. The ring of integers of K is $\O_K=\Z[\omega]$. Via the mapping $\C\rightarrow \R^2$ as $a+bi\mapsto (a,b)$, we may identify a vector $\bu+\omega \bv\in \R^n+\omega\R^n$ in $\C^n$ with a vector $(\bu-\frac 12 \bv,\frac{\sqrt{3}}{2}\bv)$ in $\R^{2n}$. So $\O_K^n$ can be regarded as a subset of $\R^{2n}$. If we define the length $\norm{\c}$ of the complex vector $\c=(a_1+b_1i,\cdots,a_n+b_ni)$ $(a_i,b_i\in\R)$ as $\sqrt{\sum_{i=1}^n(a_i^2+b_i^2)}$, then it is obvious that $\norm{\bu+\omega\bv }=\norm{(\bu-\frac 12 \bv,\frac{\sqrt{3}}{2}\bv)}$, where the second one is the Euclidean length of the vector in $\R^{2n}$.

Let $P\subseteq\Z^n$ be a packing in $\R^n$. The minimum Euclidean distance, determinant of $P\subseteq\R^n$ and $\LP=P+\omega P\subseteq\R^{2n}$ have the following relations.

\begin{lemma}[{\cite[Proposition 2.2]{xing2008dense}}]\label{degree2distance}
The minimum Euclidean distance $$d_E(P+\omega P)=d_E(P).$$
\end{lemma}
\begin{lemma}[{\cite[Proposition 2.6(i)]{xing2008dense}}]\label{degree2discriminant}
The determinant $$\det(P+\omega P)=\left(\dfrac {\sqrt 3}{2}\right)^n\left(\det(P)\right)^2.$$
\end{lemma}
Here $K$ is a totally complex field and $\O_K$ is a principal ideal domain. Given a non-zero prime ideal $\iP=(t_{\iP})$ with absolute norm $Q:=N(\iP)=
\left|\Norm_{K/\Q}(t_{\iP})\right|$, we can consider a special packing $$t_{\iP}\cdot\LP:=\left\{(t_{\iP}\alpha_1,t_{\iP}\alpha_2,\cdots,t_{\iP}\alpha_n)\in\O_K^n:(\alpha_1,\alpha_2,\cdots,\alpha_n)\in\LP\right\}.$$

 From algebraic number theory (see \cite{neukirch1999algebraic}), we know that the residue class field $\F_{\iP}=\O_K/\iP$ is isomorphic to the finite field $\F_Q$. Let $\beta_1=0,\beta_2,\cdots,\beta_Q$ be $Q$ elements of $\O_K$ such that $$\beta_1\mod \iP,\cdots,\beta_Q\mod \iP$$ represent the $Q$ distinct elements in $\F_{\iP}$.
 In the following discussion, we take the alphabet set of $Q$-ary codes to be $S=\{\beta_1,\cdots,\beta_Q\}$. In this way, the codes can be regarded as a finite subset of $\O_K^n$.

We take a family of $Q$-ary codes $\left\{C_i=(n,M_i,\geqslant Q^{\ell-i}d_E^2(\LP)\right\}_{i=0}^{\ell-1}$. The following lemma offers the concatenating method of the codes with the packing $t_{\iP}^\ell\cdot \LP\subseteq\O_K^n$. Note that the concatenation is just the sumset of the subsets in $\O_K^n$.

\begin{lemma}[{\cite[Corollary 2.4]{xing2008dense}}]\label{concatenate}
  Given a non-zero prime ideal $\iP=(t_{\iP})$ of $K=\Q(\sqrt{-3})$ such that $Q=N(\iP)=
|\Norm_{K/\Q}(t_{\iP})|$, let
\begin{enumerate}[(i)]
\item $\LP\subseteq \O_K^n$ be a packing in $\R^{2n}$;
\item $\mathcal{C}=\left\{C_i=\left(n,M_i,d_{C_i}\right)\right\}_{i=0}^{\ell-1}$ be a family of $Q$-ary codes, where the alphabet set of $C_i$ is $S$, and $d_{C_i}\geqslant Q^{\ell-i}d_E^2(\LP)$. In addition, for each $0\leqslant i\leqslant \ell-1$, $C_i$ contains zero codeword.
\end{enumerate}
  Then the concatenation $C_0+t_{\iP}C_1+\cdots+t_{\iP}^{\ell-1}C_{\ell-1}+t_{\iP}^{\ell}\cdot\LP$ is a subset of $\O_K^n$, which is defined as
   \begin{eqnarray*}
     \left\{\sum_{i=0}^{\ell-1}t_{\iP}^i\mathbf{c}_i+t_{\iP}^\ell\mathbf{p}:\mathbf{c}_i\in C_i\ \text{for all $0\leqslant i\leqslant \ell-1$,}\ \mathbf{p}\in\LP \right\}.
   \end{eqnarray*}
   It can be regarded as a packing in $\R^{2n}$ with density at least $\Delta(\LP)\cdot\prod_{i=0}^{\ell-1}M_i $. Equivalently, the density exponent $$\lambda\left(C_0+t_{\iP}C_1+\cdots+t_{\iP}^{\ell-1}C_{\ell-1}+t_{\iP}^{\ell}\cdot\LP\right)\geqslant \lambda(\LP)+ \dfrac{1}{2n}\sum_{i=0}^{\ell-1}\log_2(M_i ).$$
\end{lemma}
\begin{proof}
  First consider the case $\ell=1$, which only concatenates one code $C$ with the packing $t_{\iP}\cdot\LP$. Then use induction to get the general result. For the details, readers may refer to \cite[Corollary 2.4]{xing2008dense}.
\end{proof}

Note that the requirement that each code concatenated contains zero codeword is necessary for Proposition 2.3 and Corollary 2.4 of \cite{xing2008dense} as the proof requires that any codeword in $C$ has Hamming weight not less than the minimum Hamming distance of $C$.


\section{Remarks on the Asymptotic Properties of Xing's Construction}\label{sec:remark on Xing}
Based on Lemma \ref{concatenate} (Xing's construction), a direct idea for constructing asymptotically good packing family is to take $\LP$ as $\O_K^n$ and let $n$ tend to $\infty$. The result somehow is not included in Xing's paper \cite{xing2008dense}. Here we exhibit it as a benchmark. Moreover, in order to highlight our innovation and contribution, we briefly recall Xing's asymptotically good packing family.

\subsection{Asymptotically Good Packing Family Derived from $\O_K^n$}
Based on the GV bound \eqref{eqn:GV bound}, for $0\leqslant i\leqslant\ell-1$, we can choose $Q$-ary codes \begin{eqnarray*}\label{def:asymptotic single code}
  C^{(\ell)}_{i}=\left(n_\ell,Q^{n_\ell R^{(\ell)}_i},Q^{\ell-i}\right),\quad \text{where $n_\ell=Q^\ell$},
\end{eqnarray*} such that
the rate $$R^{(\ell)}_i\geqslant R_{GV}\left(Q,\varrho^{(\ell)}_i\right)=1-H_Q\left(\varrho^{(\ell)}_i\right),$$where the relative minimum distance $$\varrho^{(\ell)}_i=\dfrac{Q^{\ell-i}}{n_\ell}=\dfrac {1} {Q^i}
.$$
\begin{proposition}\label{prop:ring of integers}

   Set the packing $\LP$ in Lemma \ref{concatenate} as $\O_K^{n_\ell}$, where $n_\ell=Q^\ell$. Then the asymptotic density exponent $\lambda(\mathcal{F})$ of the packing family
   $$\mathcal{F}=\left\{C_0^{(\ell)}+t_{\iP}C_1^{(\ell)}+\cdots+t_{\iP}^{\ell-1}C_{\ell-1}^{(\ell)}+t_{\iP}^{\ell}\cdot\O_K^{n_\ell}\right\}_{\ell\rightarrow \infty}$$   satisfies
  \begin{eqnarray}\label{eqn:asymptotic of O_K}
    \lambda(\mathcal{F})\geqslant -1+\frac 12\log_2 2\pi e-\frac 14\log_2 3-\frac 1 2\log_2 Q\cdot \sum_{i=0}^{\ell-1} H_Q'(1/Q^i),
  \end{eqnarray}
  where $H_Q'(\varrho)=H_Q(\varrho)$ for $0<\varrho<\dfrac {Q-1}{Q}$ and $H_Q'(\varrho)=1$ for $\dfrac{Q-1}{Q}\leqslant \varrho\leqslant 1$.
\end{proposition}

\begin{proof}
   From Lemma \ref{degree2distance} and Lemma \ref{degree2discriminant}, we know for each $n\in \Z_{\geqslant 1}$,
   \begin{eqnarray*}
     d_E(\O_K^n)=1,\quad \text{and} \quad \det(\O_K^n)=\left(\dfrac{\sqrt{3}}{2}\right)^n.
   \end{eqnarray*} Hence for $n=Q^\ell$, the density exponent of $\O_K^n$ satisfies
  \begin{eqnarray*}
    \lambda(\O_K^n)=\dfrac{1}{2n}\log_2\dfrac{\left(1/2\right)^{2n}V_{2n}}{\left(\sqrt{3}/2\right)^n}
     =-1+\frac 12\log_2 2\pi e-\frac 14\log_2 3-\frac \ell 2\log_2 Q .
  \end{eqnarray*}
  From Lemma \ref{concatenate}, we get \eqref{eqn:asymptotic of O_K}.

\end{proof}
\begin{remark}\label{remark:ring of integers}
There is no clear monotonicity of the lower bound \eqref{eqn:asymptotic of O_K}. We apply software \Magma{} V2.20-7 \cite{MR1484478,MS} to list all prime numbers within $100$. Let $p$ run through the list and choose one splitting prime ideal of $p$ as $\iP$. $Q$ is the norm of $\iP$. Set $\ell=1000$, which is sufficiently large to approximate the limit on the level of \Magma{} precision. The best result is
\begin{verbatim}
[Ring of integers in K]  -1.27196767512213615952191570262
  when Q=4 norm of prime ideal lying over 2.
  \end{verbatim}
  On the whole, the result will go worse when the prime number $p$ increases.
\end{remark}

\subsection{Xing's Asymptotically Good Packing Family}
Xing offered one method (Theorem 3.4 of \cite{xing2008dense}) to improve the asymptotic bound in Remark \ref{remark:ring of integers}. We retest the asymptotic density exponent of Xing's construction first.

\begin{remark}\label{remark:result of Xing}
In our experiment, we test the prime numbers within $50$, run through $z=1/Q$ to $(Q-1)/Q$ by $1/10000$. The best result of Xing's construction is
\begin{verbatim}
[Xing] -1.26532181415209410650824899158
 when z=3049/10000, Q= 4 norm of prime ideal lying over 2.
\end{verbatim}
\end{remark}
Note that $z\approx0.3049$ is the computational optimal estimate. Suppose the real optimal is $z_0$. We briefly sketch Xing's construction then.

 Instead of $\O_K^n$, Xing considered the packing $\LP_x\subseteq \O_K^n$ such that $d_E(\LP_x)$ is a integer $x$. Set $$\mathcal{F}_x=\left\{C_0^{(x)}+t_{\iP}C_1^{(x)}+\cdots+t_{\iP}^{\ell-1}C_{\ell-1}^{(x)}+t_{\iP}^{\ell}\cdot\LP_x\right\}_{n\rightarrow \infty},$$ where $\ell=\left\lfloor\log_Q(n/x)\right\rfloor$. One lower bound of its asymptotic density exponent is given in Theorem 3.2 of \cite{xing2008dense}.

If there exist an integer $x$ such that exactly $\dfrac{x}{Q^{\left\lceil\log_q x\right\rceil}}=z_0$, then the packing family $\mathcal{F}_x$ can obtain the optimal bound of \cite[Theorem 3.4]{xing2008dense}. Otherwise, we can select a sequence of integers $\left\{x_k\right\}$ such that $\lim_{k\rightarrow \infty}\dfrac{x_k}{Q^{\left\lceil\log_q x_k\right\rceil}}=z_0$, and then use diagonal argument to group a new family $\mathcal{F}'$ from $\{\mathcal{F}_{x_k}\}_{k\rightarrow \infty}$, where the $k$-th member of $\mathcal{F}'$ is the $k$-th member of $\mathcal{F}_{x_k}$. Then the new family $\mathcal{F}'$ can obtain the optimal bound of \cite[Theorem 3.4]{xing2008dense}.


\section{New Method to Construct Asymptotically Good Family}\label{sec:Our concatenation}\label{sec:our general method}
 In Xing's construction, the number of codes increases to $\infty$ as $n$ tends to $\infty$. He concatenated these codes to certain families of packings, which are not asymptotically good. In this paper, we exhibit a new constructing method that we concatenate finitely many codes to asymptotically good packing families. In particular, our method can obtain some packing families which are derived from, but denser than, the multiplicative lattice packing families. The results will be explicitly shown in next section.

Suppose we have an asymptotically good lattice packing family $\mathcal{F}=\left\{{L}_n\right\}_{n\rightarrow\infty }$ in $\R^n$ with $d_E(L_n)\geqslant c\sqrt{n}$ for some constant $c>0$.

Let $Q$ be the norm of one prime ideal $(t_{\iP})$. Set $\ell=\left\lfloor \log_Q \dfrac{(Q-1)}{c^2Q}\right\rfloor$. Thus $Q^\ell\cdot c^2\leqslant \dfrac{Q-1}{Q}$. Based on the GV bound \eqref{eqn:GV bound}, for $0\leqslant i\leqslant \ell-1$, we can choose $Q$-ary codes $$C_i^{(n)}=\left(n,Q^{nR_i^{(n)}},\left\lceil Q^{\ell-i}\cdot c^2n \right\rceil\right)$$ such that the rate
\begin{eqnarray*}
   R_i^{(n)}\geqslant R_{GV}\left(Q,\varrho_i^{(n)}\right)=1-H_Q\left(\varrho_i^{(n)}\right),
\end{eqnarray*}
where the relative minimum distance
\begin{eqnarray*}
  \lim_{n\rightarrow\infty} \varrho_i^{(n)}= Q^{\ell-i}\cdot c^2.
\end{eqnarray*}

\begin{proposition}\label{prop:general exponent}
  We can concatenate $\ell=\left\lfloor\log_Q\dfrac {Q-1}{c^2Q}\right\rfloor$ $Q$-ary codes $$\left\{C_i^{(n)}=\left(n,Q^{nR_i^{(n)}},\left\lceil Q^{\ell-i}\cdot c^2n \right\rceil\right)\right\}_{i=0}^{\ell-1}$$ to $\LP_n:=L_n+\omega L_n$. The asymptotic density exponent of the new packing family
  $$\mathcal{H}=\left\{C_0^{(n)}+t_{\iP}C_1^{(n)}+\cdots+t_{\iP}^{\ell-1}C_{\ell-1}^{(n)}+t_{\iP}^{\ell}\cdot \LP_n\right\}$$
  satisfies
  \begin{eqnarray*}
    \lambda(\mathcal{H})&\geqslant& \dfrac 12\log_2\dfrac{c^2\pi e}{2\sqrt{3}}-\dfrac {1}{n} \log_2 \det(L_n)+\dfrac 12 \log_2 Q \sum_{i=0}^{\ell-1}\left(1-H_Q\left(Q^{\ell-i} c^2\right)\right).
  \end{eqnarray*}
\end{proposition}
 \begin{proof} From the definition of asymptotic density exponent, we have
    \begin{eqnarray*}
      \lambda(\mathcal{H})&\geqslant&\limsup_{n\rightarrow \infty}\dfrac{1}{2n}\log_2\dfrac{\left(c\sqrt{n}\right)^{2n}V_{2n}\prod_{i=0}^{\ell-1}Q^{nR_i^{(n)}}}{2^{2n}\det(\LP_n)}\\
      &=&\limsup_{n\rightarrow \infty}\log_2\dfrac c2+\dfrac 12 \log_2n-\dfrac 12 \log_2 n+\dfrac 12\log_2\pi e\\
      &&\qquad \qquad -\frac12 \log_2\frac{\sqrt{3}}{2}-\dfrac {1}{n} \log_2 \det(L_n)+\dfrac 12 \log_2Q\sum_{i=0}^{\ell-1}R_i^{(n)}\\
      &\geqslant&\dfrac 12\log_2\dfrac{c^2\pi e}{2\sqrt{3}}-\dfrac {1}{n} \log_2 \det(L_n)+\dfrac 12 \log_2 Q \sum_{i=0}^{\ell-1}\left(1-H_Q\left(Q^{\ell-i} c^2\right)\right).
    \end{eqnarray*}
 \end{proof}

\section{Concatenation with Multiplicative Lattices}\label{concatenation with multiplicative lattices}
Rosenbloom and Tsfasman \cite{Multiplicative1990} introduced two kinds of multiplicative lattices in global fields, that is, principal lattices and congruence lattices. In this paper, we only use the ones in function fields, where both of principal and congruence lattices are full rank sublattices of $A_{n-1}=\{\mathbf{x}\in \Z^n|\ \sum x_i=0\}$. They lead to asymptotically good packing families. In this section, we proceed with our new concatenating method introduced in Section \ref{sec:our general method} to improve the asymptotic density exponent derived from multiplicative lattice packings.
\subsection{Principal Lattices and Congruence Lattices}

We recall the definition of principal lattices from \cite{Multiplicative1990} first. Let $k=\F_q$ and $K=k(X)$, where $X/k$ be a smooth proper curve of genus $g$. Take a nonempty set $S=\{P_1,P_2,\cdots,P_n\}\subseteq X(k)$, $n=|S|$, and let
\[U_S=\left\{f\in K^* |\ f\ \text{is a unit outside}\ S\right\}.\]
Let $\Div _S(X)$ be the group of divisors supported in $S$, $\Div^0_S(X)\subseteq \Div_S(X)$ the subgroup of degree zero divisors, $\Pr_S(X)$ the subgroup of principal divisors, and let $J_X=\Div^0(X)/\Pr(X)$ denote the Jacobian of $X$. The properties  of these groups can be found in \cite[Chapter 1]{stichtenoth2009algebraic}.

There is a natural map
\begin{eqnarray*}
  \phi: U_S&\rightarrow& \Div_S(X)\simeq \Z^n\\
   f&\mapsto&\div(f),
\end{eqnarray*}
where $\div (f)$ is the principal divisor of $f$.
The principal lattice is defined as $L_S:=\Pr_S(X)=\phi(U_S)$, which is a sublattice of $A_{n-1}$. The parameters of $L_S$ are

\begin{lemma}[{\cite[Lemma 1.1]{Multiplicative1990}}]\label{parameter of principle lattice}
\begin{enumerate}[(i)]
\item  $\rank L_S=n-1$;
\item  $\det L_S\leqslant \sqrt{n}\cdot |J_X(k)|$;
\item  $d_E(L_S)\geqslant \min _{f\in U_S\setminus k^*}\sqrt{2\cdot \deg f}$.
\end{enumerate}
\end{lemma}

Furthermore, let $D$ be a positive divisor on $X$, $D=\sum a_iQ_i$, $r_i=\deg Q_i$, $N(Q_i)=q^{r_i}$, $a=\deg D=\sum a_ir_i$. Here we assume $S\cap \supp(D)=\emptyset$.
Then the congruence lattice is defined as $L_{S,D}:=\phi(U_{S,D})$, where
\[U_{S,D}=\left\{f\in U_S: f\equiv 1\mod D\right\}.\]

The parameters of $L_{S,D}$ are

\begin{lemma}[{\cite[Lemma 2.2]{Multiplicative1990}}]
\begin{enumerate}[(i)]
\item $\rank L_{S,D}=n-1$;
\item $\det L_{S,D}\leqslant \sqrt{n}\cdot|J_X(k)|\cdot\dfrac{q^a}{q-1}\cdot\prod(1-q^{-r_i})$;
\item $d_E(L_{S,D})\geqslant \sqrt{2a}$.
\end{enumerate}
\end{lemma}

\subsection{Lattice Dimension Augmentation for Full Rank Sublattices of $A_{n-1}$}\label{Dimension augment}

We know the rank of $A_{n-1}$ is $n-1$. Now we want to apply our concatenating method on certain full rank sublattices of $A_{n-1}\subseteq \Z^{n}$. First we need introduce a dimension augmentation method to make the lattices have rank $n$ without much loss in the parameters.

For any full rank sublattice $L$ of $A_{n-1}$, the $\R$-linear span of $L$ is
\begin{eqnarray*}
  V=\{(x_1,x_2,\cdots,x_{n-1},x_n)\in\R^n:x_1+x_2+\cdots+x_{n}=0\}.
\end{eqnarray*} We add one extra row vector $\e_n=(0,0,0,0,\cdots,0,\chi)$ to the generator matrix of $L$, where $\chi\in\Z\setminus\{0\}$. The resulting matrix generates a rank $n$ lattice in $\R^n$, which is denoted by $B$ and called the augmented lattice of $L$.

The distance from the point $(0,0,0,0,\cdots,0,\chi)$ to the hyperplane $V$ is $\dfrac{\chi}{\sqrt{n}}$. Thus the minimum Euclidean distance of $B$ satisfies
\begin{eqnarray*}
  d_E(B)\geqslant \min\left\{d_E(L),\dfrac{\chi}{\sqrt{n}}\right\}.
\end{eqnarray*}

\subsection{Concatenation with Principal Lattices}

Now set $S=X(k)$, and use the same estimation $\deg f\geqslant \dfrac{|X(k)|}{q+1}$ as \cite{Multiplicative1990}. Thus the minimum Euclidean distance of $L_S=L_{X(k)}$ satisfies $d_E(L_{X(k)})\geqslant\sqrt{\dfrac{2n}{q+1}}$, where $n=|X(k)|$.

We add the row vector $(0,0,0,0,\cdots,0,n)$ to the generator matrix of $L_{X(k)}$ and obtain a rank $n$ lattice $B_{X(k)}$ in $\Z^n$. The parameters of $B_{X(k)}$ are

\begin{proposition}\label{prop:augmented principal lattice}
  \begin{enumerate}[(i)]
\item  $\rank B_{X(k)}=n$;
\item  $\det B_{X(k)}=\dfrac{n}{\sqrt{n}}\cdot\det(L_{X(k)})\leqslant n\cdot |J_X(k)|$;
\item  $d_E(B_{X(k)})\geqslant \min\left\{d_E(L_{X(k)}),\dfrac{n}{\sqrt{n}}\right\}\geqslant \sqrt{\dfrac{2n}{q+1}} $.
\end{enumerate}
\end{proposition}
\begin{proof}
  (i)(iii) are directly from the dimension augmentation method. For (ii), as the determinant of a lattice is just the volume of the fundamental region of the lattice, and the distance from the point $(0,\cdots,0,n)$ to the $\R$-linear span of $L_{X(k)}$ is $\dfrac {n}{\sqrt{n}}$, we get the determinant of $B_{X(k)}$ is $\dfrac {n}{\sqrt{n}}\cdot \det(L_{X(k)})$. Following lemma \ref{parameter of principle lattice}, we get the result.
\end{proof}

We employ the same families of curves as \cite{Multiplicative1990}:
 For $q$ is an even power of a prime, there exist families of curves $X/k$ of growing genus $g(X)$ such that $\lim \dfrac{|X(k)|}{g(X)}=\sqrt{q}-1$. Moreover, such families satisfy
\[|J_X(k)|\sim q^{g(X)}\left(\frac{q}{q-1}\right)^{|X(k)|}.\]
The proof of the estimation can be found in the Appendix of \cite{Multiplicative1990}. The following lemma characterizes that the corresponding augmented principal lattices lead to asymptotically good packing families.

\begin{lemma}\label{lem:original principal}
A family of curves $X/k$ with $\lim\dfrac{|X(k)|}{g(X)}=\sqrt{q}-1$ yields
a family of augmented principal lattices $\FF_0=\left\{B_{X(k)}^{(N)}\subseteq \R^N\right\}$ with rank $N=|X(k)|\rightarrow \infty$ and
\[\lambda(\FF_0)\geqslant \log {\sqrt{\pi e}}-\log{\dfrac{\sqrt{q+1}}{q-1}}-\dfrac{\sqrt{q}}{\sqrt{q}-1}\log{q}.\]

\end{lemma}
\begin{proof}
   Note that $\lim_{N\rightarrow \infty}\dfrac 1N\log_2 N=0$. The proof is straightforward from the definition of asymptotic density exponent and Proposition \ref{prop:augmented principal lattice}. It is also similar to the proof of {\cite[Theorem 1.2]{Multiplicative1990}}.
\end{proof}

Note that the bound in Lemma \ref{lem:original principal} is exactly the one of principal lattices \cite[Theorem 1.2]{Multiplicative1990}. This means the dimension augmentation do not harm the good asymptotic properties of the original lattices. Meanwhile, we put it here as a reference to compare with the following Proposition \ref{prop:concatenation with principal}. The difference is the advantage of our concatenating method.

As $d_E\left(B_{X(k)}^{(N)}\right)\geqslant \sqrt{\dfrac{2}{q+1}}\cdot \sqrt{N}$, we can proceed with the method introduced in Section \ref{sec:Our concatenation}. We denote $\LP^{(N)}_{X(k)}=B_{X(k)}^{(N)}+\omega B_{X(k)}^{(N)}$ and get the following proposition.

\begin{proposition}\label{prop:concatenation with principal}
A family of curves $X/k$ with $\lim\dfrac{|X(k)|}{g(X)}=\sqrt{q}-1$ and families of $Q$-ary codes $$\left\{C_i^{(N)}=\left(N,Q^{NR_i^{(N)}},\left\lceil Q^{\ell-i}\cdot \dfrac{2N}{q+1} \right\rceil\right)\right\}_{i=0}^{\ell-1}$$ with $\ell=\left\lfloor\log_Q\dfrac{(Q-1)(q+1)}{2Q}\right\rfloor$ and the rate
\begin{eqnarray*}
  \lim_{N\rightarrow\infty} R_i^{(N)}\geqslant R_{GV}\left(Q,Q^{\ell-i}\cdot \dfrac{2}{q+1}\right)=1-H_Q\left(Q^{\ell-i}\cdot \dfrac{2}{q+1}\right),
\end{eqnarray*}
yield a packing family
$$\FF_{Q,q}=\left\{C_0^{(N)}+t_{\iP}C_1^{(N)}+\cdots+t_{\iP}^{\ell-1}C_{\ell-1}^{(N)}+t_{\iP}^{\ell}\cdot\LP_{X(k)}^{(N)}\subseteq \R^{2N}\right\}_{N\rightarrow \infty}$$
with $N=|X(k)|\rightarrow \infty$ and
\begin{eqnarray}\label{eqn:improvement of principal lattice}
  \lambda(\FF_{Q,q})&\geqslant&\log {\sqrt{\pi e}}-\log{\dfrac{\sqrt{q+1}}{q-1}}-\dfrac{\sqrt{q}}{\sqrt{q}-1}\log{q}\nonumber\\
  &&\qquad -\frac 14\log_23+\dfrac 12 \log_2 Q \sum_{i=0}^{\ell-1}\left(1-H_Q\left(\dfrac {2Q^{\ell-i}}{q+1}\right)\right).
\end{eqnarray}

\end{proposition}

\begin{proof}
  From Proposition \ref{prop:general exponent} and \ref{prop:augmented principal lattice}.
\end{proof}

\begin{remark}\label{remark:result on principal}
  There is no clear monotonicity of the lower bound \eqref{eqn:improvement of principal lattice}. We apply software \Magma{} V2.20-7 \cite{MR1484478,MS} to list all prime numbers within $100$. Let $p_1$ run through the list and choose one splitting prime ideal of $p_1$ as $\iP$. $Q$ is the norm of $\iP$. Let $p_2$ run through the list and let $r$ run through the even numbers from $2$ to $250$. Take $q=p_2^r$. The best output in the experiment is given as

\small{
\begin{verbatim}
[Improvement on Principal Lattices]
 -1.26532182282965944267554218804
 when Q=4 norm of prime ideal lying over 2; q=59^28.
 Lattice packing contributes: -81.2061477310654255659655563902;
 l= 81 Concatenated codes contributes:
                               79.9408259082357661232900142022.
\end{verbatim}
}
\end{remark}

The above output shows that the optimal result in our experiment is $\lambda\geqslant -1.26532182283 $ when $Q=4,q=59^{28}$, which is better than $-1.87$ from principal lattices.  Note that the last two statements show the contributions from augmented principal lattices and concatenated codes respectively to the asymptotic density exponent. In Section \ref{sec:comparison}, we will use the componential contributions to compare our results on concatenations from principal lattices and congruence lattices.

\subsection{Concatenation with Congruence Lattices}

Similarly as last subsection, we set $S=X(k), n=|X(k)|$, and add the row vector $(0,0,0,0,\cdots,0,n)$ to the generator matrix of $L_{S,D}=L_{X(k),D}$ and obtain a rank $n$ lattice $B_{X(k),D}$ in $\Z^n$. The parameters of $B_{X(k),D}$ are
\begin{proposition}\label{prop:augmented congruence lattices}
  \begin{enumerate}[(i)]
\item  $\rank B_{{X(k)},D}=n$;
\item  $\det B_{{X(k)},D}=\dfrac{n}{\sqrt{n}}\det(L_{{X(k)},D})\leqslant n\cdot|J_X(k)|\cdot\dfrac{q^a}{q-1}\cdot\prod(1-q^{-r_i})$;
\item  $d_E(B_{{X(k)},D})\geqslant\min\left\{d_E(L_{{X(k)},D}),\dfrac{n}{\sqrt{n}}\right\}$.
\end{enumerate}
\end{proposition}

We consider the same families of curves as principal lattices and further choose divisors in such a way that
\begin{eqnarray*}
  \lim \dfrac{\deg D}{|X(k)|}=\dfrac{y}{2\ln q},\ \text{where } 0<y\leqslant 1.
\end{eqnarray*} Note that $\lim \dfrac{\deg D}{|X(k)|}=\dfrac{1}{2\ln q}$ is adopted in \cite{Multiplicative1990}, while here we loosen the requirement for our construction. The following lemma characterizes that the corresponding augmented congruence lattices lead to asymptotically good packing families.

\begin{lemma}
A family of curves $X/k$ with $\lim\dfrac{|X(k)|}{g(X)}=\sqrt{q}-1$ and positive divisors with $\lim \dfrac{\deg D}{|X(k)|}=\dfrac{y}{2\ln q}$ yield
 a family of augmented congruence lattices $\FF_0'=\left\{B_{X(k),D}^{(N)}\subseteq \R^N\right\}$ with rank $N=|X(k)|\rightarrow \infty$ and
\begin{eqnarray*}
  \lambda(\FF_0')&\geqslant& \log_2\sqrt{\frac{\pi}{2}}-\frac 12\log_2\left(\ln q\right)- \dfrac{\sqrt q}{\sqrt q-1}\log_2 q\\
&\ &+\log_2\left(q-1\right)+\frac 12\log_2 y+\frac{1-y}{2}\log_2 e.
\end{eqnarray*}

\end{lemma}
\begin{proof}
    Similar to the proof of Lemma \ref{lem:original principal} and {\cite[Theorem 2.3]{Multiplicative1990}}.
\end{proof}

As for sufficiently large $N=|X(k)|$,  $d_E\left(B_{X(k),D}^{(N)}\right)\geqslant \sqrt{\dfrac{y}{\ln q}}\cdot \sqrt{N}$, we can proceed with the method introduced in Section \ref{sec:Our concatenation}. We denote $\LP^{(N)}_{X(k),D}=B_{X(k),D}^{(N)}+\omega B_{X(k),D}^{(N)}$ and get the following proposition.

\begin{proposition}\label{prop:concatenation with congruence}
A family of curves $X/k$ with $\lim\dfrac{|X(k)|}{g(X)}=\sqrt{q}-1$, positive divisors with $\lim \dfrac{\deg D}{|X(k)|}=\dfrac{y}{2\ln q}$ and families of $Q$-ary codes $$\left\{C_i^{(N)}=\left(N,Q^{NR_i^{(N)}},\left\lceil Q^{\ell-i}\cdot \dfrac{yN}{\ln q}\right\rceil\right)\right\}_{i=0}^{\ell-1}$$ with $\ell=\left\lfloor \log_Q \dfrac{(Q-1)\ln q}{yQ}\right\rfloor$ and the rate
\begin{eqnarray*}
  \lim_{N\rightarrow\infty} R_i^{(N)}\geqslant R_{GV}\left(Q,Q^{\ell-i}\cdot \dfrac {y}{\ln q}\right)=1-H_Q\left(Q^{\ell-i}\cdot \dfrac {y}{\ln q}\right),
\end{eqnarray*}
yield a packing family
$$\FF_{Q,q,y}=\left\{C_0^{(N)}+t_{\iP}C_1^{(N)}+\cdots+t_{\iP}^{\ell-1}C_{\ell-1}^{(N)}+t_{\iP}^{\ell}\cdot\LP_{X(k),D}^{(N)}\subseteq \R^{2N}\right\}_{N\rightarrow \infty}.$$
with $N=|X(k)|\rightarrow \infty$ and
\begin{eqnarray}\label{eqn:improvement of congruence lattice}
  \lambda(\FF_{Q,q,y})&\geqslant&\log_2\sqrt{\frac{\pi}{2}}-\frac 12\log_2\left(\ln q\right)- \dfrac{\sqrt q}{\sqrt q-1}\log_2 q\nonumber\\
&\ &+\log_2\left(q-1\right)+\frac 12\log_2 y+\frac{1-y}{2}\log_2 e\nonumber\\
  &\ & -\frac 14\log_23+\dfrac 12 \log_2 Q \sum_{i=0}^{\ell-1}\left(1-H_Q\left(\dfrac {yQ^{\ell-i}}{\ln q}\right)\right).
\end{eqnarray}

\end{proposition}

\begin{proof}
  From Proposition \ref{prop:general exponent} and \ref{prop:augmented congruence lattices}.
\end{proof}

\begin{remark}\label{remark:result on congruence}
  There is no clear monotonicity of the lower bound \eqref{eqn:improvement of congruence lattice}. We design the computational experiments in \Magma{} V2.20-7 \cite{MR1484478,MS} as follows:
\begin{itemize}
\item List all prime numbers within $60$. Let $p_1$ run through the list and choose one splitting prime ideal of $p_1$ as $\iP$. $Q$ is the norm of $\iP$. Let $p_2$ run through the list and let $r$ run through the even numbers from $2$ to $100$. Take $q=p_2^r$.
\item Set $y$ from $0.1$ to $1$ by $0.01$. Then we find the good result comes when $y=0.1$. Next set $y$ from $0.01$ to $0.2$ by $0.0001$. Then repeatedly increase the decimal places to get $y$ for better results.
    \end{itemize}
    We can not run through all prime numbers and all possible values for $y$. The best output in the experiment is given as
\small{
\begin{verbatim}
[Improvement on Congruence Lattices]
 -1.26532181404273379250349262485
 when Q=4 norm of prime ideal lying over 2;
 q=11^94; y=1/4000000000.
 Lattice packing contributes: -19.2060002184860472925950737917;
 l= 19 Concatenated codes contributes:
                               17.9406784044433135000915811668.
\end{verbatim}
}
\end{remark}
The above output shows that the optimal result in our experiment is $\lambda\geqslant -1.26532181404 $ when $Q=4,q=11^{94},y=2.5\times 10^{-10}$, which is better than $-1.39$ from congruence lattices.
 \section{Comparison}\label{sec:comparison}

 In Rosenbloom
and Tsfasman's construction \cite{Multiplicative1990}, congruence lattices lead to an asymptotically good family with $\lambda\geqslant -1.39$, which is better than $\lambda\geqslant -1.87$ of the packing family from principal lattices. However, through our concatenating method, the family derived from congruence lattice has bound only slightly better than the one from principal lattices, while both of the bounds on $\lambda$ are quite similar with Xing's result~\cite{xing2008dense}. It deserves a comparison here.

First we take the case $Q=4,q=11^{94},y=2.5\times 10^{-10}$ as an example, which leads to the best result in the experiment in Remark \ref{remark:result on congruence}, and compare the concatenations derived from congruence (Proposition \ref{prop:concatenation with congruence}) and principal (Proposition \ref{prop:concatenation with principal}) lattices respectively. Let $\ell$ denote the number of concatenated codes and $c$ denote the coefficient used in the bound $d_E(B)\geqslant c\sqrt{n}$, while $c=\sqrt{\dfrac{2}{q+1}}$ in principal case and $c=\sqrt{\dfrac{y}{\ln q}}$ in congruence case. We disassemble the density exponents by contributions from lattice packing and concatenated codes. The numerical results are listed in Table 1.

 \begin{table}[htbp]
\caption{Componential Contributions to Asymptotic Density Exponent}
\centering
\begin{threeparttable}
\begin{tabular}{|c|c|c|}

\hline
&Based on Principal lattices&Based on  Congruence lattices\\
 \hline

 \multirow{2}{*}{Lattice}&$c=1.60346245499\times 10^{-49}$&$c=1.05315179371\times 10^{-6}$\\
\cline{2-3}
&$-161.44243111595$&$ -19.20600021848$\\
\hline
\multirow{2}{*}{Codes}&$\ell=161$&$\ell=19$\\
\cline{2-3}
&$160.15877344941$&$17.94067840444$\\
\hline
$\lambda\geqslant$ &$-1.28365766654$&$-1.26532181404$
\\
\hline
\end{tabular}

      \end{threeparttable}

\end{table}
From the table, we can find that for same $q$, the density contribution from principal lattices is less than congruence lattices, which is consistent with the result in \cite{Multiplicative1990}. However, the bound $c\sqrt{n}$ on the minimum Euclidean distance of principal lattices are much smaller than congruence lattices, which leads to the benefit that we can concatenate more codes with it. More codes contribute more in the density exponent. As a result, the bounds on $\lambda$ are similar.

Compared with Xing's construction, as introduced in Section \ref{sec:remark on Xing}, we concatenate finitely many codes with asymptotically good packing families, while Xing concatenated approximately infinitely many codes with asymptotically bad packing families. The two constructions are essentially different. Moreover, we also test the sequences $\left\{\log_Q\left\lceil c\sqrt{n}\right\rceil-\left\lceil\log_Q \left\lceil c\sqrt{n}\right\rceil\right\rceil\right\}$, where $c$ equals the values shown in the above table. There are only few $n$'s such that the corresponding value is close to $0.3049$. Thus our constructions are different with Xing's as they do not satisfy the requirement in Xing's construction.

Based on the numerical results in Remark \ref{remark:result of Xing}, \ref{remark:result on principal}, \ref{remark:result on congruence}, our packing family derived from congruence lattices has slightly better density exponent than the one from principal lattices, and the one from Xing's construction.

\section{Conclusion}\label{sec:conclusion}
In this paper, we explicitly construct two asymptotically good packing families. The main technique is to concatenate families of codes attaining GV-bound with multiplicative lattices. Our constructions improve the bounds on the asymptotic density exponent of packing families derived from multiplicative lattices.
Moreover, concatenation method offers a channel to unify the constructions of packing from different disciplines, such as curves over finite fields and coding theory, which are the source materials in present paper. Furthermore, we may generalize the construction based on arbitrary number field instead of only $\Q(\sqrt{-3})$. This is left for future research to enhance the concatenating method.

\section*{Acknowledgements}
The author is sincerely grateful to his supervisors, San Ling and Chaoping Xing, for introducing him to this topic, especially for the invaluable suggestions and comments from Chaoping Xing which make the author's initial idea become mature. The author also thanks the reviewers for their very careful reading.





\end{document}